\documentclass[11pt]{article}

\usepackage{amsmath, amsfonts, amsthm, graphicx}
\usepackage[top=21mm, bottom=22mm, left=21mm, right=21mm]{geometry}
\usepackage{comment}
\usepackage{hyperref}
\hypersetup{
	colorlinks=true,
	linkcolor=blue,
	filecolor=magenta,      
	urlcolor=cyan,
	citecolor=blue
}
\usepackage{enumitem}
\usepackage{aliascnt} 
\usepackage[none]{hyphenat}

\newtheorem{theorem}{Theorem}[section]

\newaliascnt{lemma}{theorem}
\newtheorem{lemma}[lemma]{Lemma}
\aliascntresetthe{lemma}
\newaliascnt{claim}{theorem}
\newtheorem{claim}[claim]{Claim}
\aliascntresetthe{claim}
\newaliascnt{question}{theorem}

\aliascntresetthe{question}

\usepackage{cleveref}
\crefname{lemma}{lemma}{lemmas}
\Crefname{lemma}{Lemma}{Lemmas}
\crefname{claim}{claim}{claims}
\Crefname{claim}{Claim}{Claims}
\crefname{question}{question}{questions}
\Crefname{question}{Question}{Questions}

\theoremstyle{definition}
\newtheorem{definition}{Definition}

\newcommand{\typeone}{\mathbf{T1}}
\newcommand{\typetwo}{\mathbf{T2}}
\newcommand{\typethree}{\mathbf{T3}}
\newcommand{\typefour}{\mathbf{T4}}
\newcommand{\typezero}{\mathbf{T0}}
\newcommand{\eps}{\varepsilon}

\DeclareMathOperator{\tw}{tw}

\title{Ramsey theory of low-degree semialgebraic relations}
\author{Azem Adibelli\thanks{Ume\r{a} University, \emph{e-mail}: \textbf{azem.adibelli@umu.se, istvantomon@gmail.com}}
~and Istv\'an Tomon\footnotemark[1]}
\date{}

\begin{document}

\maketitle
\sloppy

\begin{abstract}
We prove that hypergraphs defined by low-degree polynomial inequalities contain large homogeneous subsets. Formally, let $H$ be an $r$-uniform hypergraph on $N$ vertices that is semialgebraic of constant description complexity, and each defining polynomial has degree at most $D$. Then $H$ contains a clique or an independent set of size $n$, where $N\leq \tw_{3D^3}(n)$.    
\end{abstract}

\section{Introduction}

Ramsey's theorem is a fundamental result in combinatorics and logic, stating that for every $n$ there exists a smallest number $N=R_r(n)$ such that every $r$-uniform $N$-vertex hypergraph contains a clique or an independent set of size $n$. It is a central topic in extremal combinatorics to understand the growth rate of $R_r(n)$, and to study how this function behaves in restricted hypergraph classes.

The original proof of Ramsey \cite{Ramsey} gave extremely poor bounds on $R_r(n)$. But soon after, Erd\H{o}s and Szekeres \cite{ESz35} and Erd\H{o}s \cite{E47} determined the right order in the case of graphs: $R_2(n)=2^{\Theta(n)}$. Finding the optimal constant factor is a notorious open problem, with several exciting recent developments \cite{CGMS,MSX}, see the survey of Morris \cite{Morris} for further discussion. For $r\geq 3$, Erd\H{o}s and Rado \cite{ER52} and Erd\H{o}s, Hajnal, and Rado \cite{EHR65} proved that $$\tw_{r-1}(\Omega(n^2))<R_r(n)<\tw_r(O(n)),$$
where the tower function $\tw_k(n)$ is defined recursively by $\tw_1(x):=x$ and $\tw_{k}(x):=2^{\tw_{k-1}(x)}$. Determining the correct height of the tower remains open, but multicolor versions of the problem indicate that the upper bound might be closer to the truth. For more recent developments on hypergraph Ramsey numbers, we refer the reader to \cite{CFS10}.

These results paint a clear picture: while we are guaranteed homogeneous sets (i.e., clique or independent set) of increasing size, there are hypergraphs that only contain extremely small such sets. This limits the applications of general Ramsey results, where one is interested in the size of homogeneous subsets of highly structured hypergraphs. Typical examples of such applications arise in geometry. For instance, the Erd\H{o}s-Szekeres theorem \cite{ESz35} states that every sequence of $N$ real numbers $a_1,\dots,a_N$ contains a monotone subsequence of length at least $\sqrt{N}$. Observe that a weaker bound of the form $\Omega(\log N)$ follows from the result $R_2(n)=2^{\Theta(n)}$:  consider the graph $G_{\text{mon}}$ on vertex set $[N]$ in which $i<j$ are joined by an edge if $a_i<a_j$, and note that a homogeneous subset in this graph is a monotone subsequence. However, the special structure of this graph ensures much larger homogeneous sets. Another classical example is the Happy Ending problem \cite{ESz35}, asking for the smallest number $N=K(n)$ such that any set of $N$ points in the plane contains $n$ points in convex position. It is a simple exercise to show that $K(n)\leq R_3(n)$, which follows by considering the $N$-vertex 3-uniform hypergraph $H_{\text{cup}}$, whose edges correspond to so called \emph{cups}. This gives a double exponential upper bound on $K(n)$, while it is now known \cite{Suk} that $K(n)=2^{n+o(n)}$. 

This motivates the question: for which natural hypergraph families can the tower-type bound $R_r(n)\leq \tw_r(O(n))$ be significantly improved? \emph{Semialgebraic hypergraphs} provide a unifying framework capable of capturing a broad spectrum of geometrically defined structures. We define an $r$-uniform hypergraph $H$ to be semialgebraic of description complexity $(d,D,m)$, if its vertices are points in $\mathbb{R}^d$ and its edge relation is determined by a Boolean combination of $m$ polynomial inequalities of degree at most $D$ (a formal definition is presented in the Preliminaries).  For example, the previously discussed graph $G_{\text{mon}}$ is semialgebraic of complexity $(1,1,1)$, while $H_{\text{cup}}$ is semialgebraic of complexity $(2,2,1)$.  Under the assumption of fixed description complexity, many problems that are hard for general hypergraphs become more tractable.

Let $R_r^{d,D,m}(n)$ denote the smallest $N$ such that any $r$-uniform $N$-vertex semialgebraic hypergraph of complexity $(d,D,m)$ contains a clique or an independent set of size $n$. Clearly, $R_r^{d,D,m}(n)\leq R_r(n)$. In the case of graphs, the function $R_2^{d,D,m}(n)$ was first studied in the foundational paper of Alon, Pach, Pinchasi, Radoi\v ci\'c and Sharir \cite{APPRS05}, where they proved that $R_2^{d,D,m}(n)\leq n^{O(1)}$. Here and later, the constant hidden by $O(.)$ and $\Omega(.)$ might depend on $r,d,D,m$, but no other parameter. Recently, sharp bounds on the exponent were proved by Tidor and Yu \cite{TY}. Note that these represent an exponential improvement compared to the trivial bound $R_2^{d,D,m}\leq R_2(n)=2^{\Theta(n)}$. By adapting the inductive approach of Erd\H{o}s and Rado \cite{ER52}, Conlon, Fox, Pach, Sudakov, and Suk \cite{CFPSS14} demonstrated that this improvement extends to the hypergraph case as well: $$R_r^{d,D,m}(n)\leq \tw_{r-1}(n^{O(1)}).$$ More strikingly, they also showed that this bound is tight: for every uniformity $r\geq 2$, there exist parameters $(d,D,m)$ such that $R_r^{d,D,m}(n)\geq \tw_{r-1}(\Omega(n))$. Consequently, the improvement provided by semialgebraic structure is only a single exponential, leaving the fundamental tower-type dependence caused by the uniformity $r$ intact. A natural direction to circumvent this is to consider how the specific parameters of the description complexity influence the growth rate of $R_r^{d,D,m}(n)$.

A beautiful result in this direction, due to Bukh and Matou\v{s}ek \cite{BM12}, establishes that $$R_r^{1,D,m}(n)\leq 2^{2^{O(n)}}$$ for some $c=c(r,D,m)$. That is, $N$-vertex semialgebraic hypergraphs of constant complexity defined over the 1-dimensional real space always contain homogeneous sets of size $\Omega(\log \log N)$, independent of the uniformity. This bound is also tight \cite{BM12,CFPSS14}. It remains an intriguing open problem whether similar phenomena holds for $d\geq 2$; specifically, whether there exists $t=t(d)$ such that $R^{d,D,m}_r(n)\leq \tw_t(O(n))$. Toward a lower bound, Eli\'a\v{s}, Matou\v{s}ek, Rold\'an-Pensado, and Safernov\'a \cite{EMRS} provided a construction showing $R_{d+3}^{d,D,m}(n)\geq \tw_{d+2}(\Omega(n))$.

Another direction is to bound the parameter $D$ denoting the degree of defining polynomials. In recent years, hypergraphs defined by linear inequalities, or equivalently, semialgebraic hypergraphs of complexity $(d,1,m)$, have received substantial interest. This is largely due to a result of Basit, Chernikov, Starchenko, Tao, and Tran \cite{BCSTT21}, which demonstrated that such hypergraphs exhibit interesting behavior with respect to Zarankiewicz-type problems, a topic which has since seen significant development \cite{CH,CKS,HMST}. The authors also coined the term \emph{semilinear hypergraph} for this family. Ramsey-type problems for semilinear graphs were studied by Tomon \cite{semilin}, while for semilinear hypergraphs, Jin and Tomon \cite{JinTomon} proved the optimal bound 
$$R^{d,1,m}_r(n)\leq 2^{n^{O(1)}}.$$ Thus, restricting the degree to $D=1$ has a similar effect on the Ramsey number as restricting the dimension to $d=1$. The main result of our paper shows that for any fixed degree $D$, the Ramsey number $R^{d,D,m}_r(n)$ can be bounded by $\tw_{t}(O(n))$, where the tower height $t=t(D)$ only depends on $D$. This provides a substantial improvement over the $R_r^{d,D,m}(n)\leq \tw_{r-1}(O(n))$ bound for hypergraphs defined by low degree polynomials. 

\begin{theorem}\label{thm:main}
     Let $H$ be an $r$-uniform semialgebraic hypergraph of description complexity $(d,D,m)$ on $N$ vertices, where $N\gg d,D,m,r$. Then $H$ contains a clique or an independent set of size $n$, where 
     $$N\leq \tw_{3D^3}(n).$$
\end{theorem}

In Subsection \ref{sect:maintheorem}, we deduce \Cref{thm:main} from a more general result, which applies to a larger class of hypergraphs. The tower height $O(D^3)$ in \Cref{thm:main} is likely far from optimal, but it cannot be replaced by a value smaller than $D-1$. This is due to the fact that any $r$-uniform semialgebraic hypergraph of description complexity $(d,D,m)$ also has description complexity $(d',r,m)$ for some $d'$ depending only on $r,d,D$. This follows from a standard application of Veronese mappings. Since it is known that there exist constructions achieving the bound $R^{d,D,m}(n)\geq \tw_{r-1}(\Omega(n))$, setting $D=r$ also shows that $R^{d',D,m}_D(n)\geq \tw_{D-1}(\Omega(n))$ with suitable~$d'$ and $m$. 

\medskip

\noindent
\textbf{Organization.} In the next section, we present our notation and a few simple auxiliary results. In Section \ref{sect:outline}, we give a rough outline of the proof of  \Cref{thm:main}. Then, in Section \ref{sect:D=2}, we present the proof of the $D=2$ case, in order to illustrate our key new ideas in a less technical manner. The detailed proof of  \Cref{thm:main} is then presented in Section \ref{sect:proof}.

\section{Preliminaries}

First, we introduce our terminology. We use mostly standard graph and set theoretic notation. We write $[n]:=\{1,\dots,n\}$. For ease of notation, we declare certain variables as constants, and then the $O(.)$ and $\Omega(.)$ notations hide factors that possibly depend on these constants, but no other variables. Moreover, if we say a quantity is sufficiently large, it is with respect to constants.

The \emph{tower function} $\tw_k(x)$ is defined as $\tw_1(x):=x$ and $\tw_k(x):=2^{\tw_{k-1}(x)}$ for $k\geq 2$. We highlight the simple identity $\tw_{k}(\tw_{\ell}(x))=\tw_{k+\ell-1}(x)$.

An \emph{ordered set} refers to any set with a complete ordering on its elements. We denote this ordering by simply $<$. Given an ordered set $A$ and integer $k$, we write (unconventionally)
$$A^{(k)}:=\{(a_1,\dots,a_k)\in A^k: a_1<\dots<a_k\}.$$
 Also, write $$A^{(\leq k)}:=\bigcup_{0\leq \ell\leq k}A^{(\ell)},$$ where we use the convention that $A^{(0)}=\{()\}$. Given $a=(a_1,\dots,a_k)\in A^k$ and $I\subset [k]$, 
$$a_I:=((a_i)_{i\in I}).$$
With slight abuse of notation, if $a\in A^{k}$ and $b\in A^{m}$, we write $(a,b)$ for $(a_1,\dots,a_k,b_1,\dots,b_m)\in A^{k+m}$

\subsection{Semialgebraic hypergraphs}
 We give a formal definition of semialgebraic hypergraphs of description complexity $(d,D,m)$, where $d$ refers to the dimension of the ambient space, $D$ is an upper bound on the (total) degree of the defining polynomials, and $m$ is an upper bound on the number of defining polynomials. 

\begin{definition}[Semialgebraic hypergraph]
Let $V\subset \mathbb{R}^d$ be an ordered set. Let $f_1,\dots,f_m: (\mathbb{R}^d)^r\rightarrow \mathbb{R}$ be polynomials of degree at most $D$, and let $\Gamma: \{\texttt{false},\texttt{true}\}^{m}\rightarrow \{\texttt{false},\texttt{true}\}$ be a Boolean formula. Define the $r$-uniform hypergraph $H$ on vertex set $V$ in which $x\in V^{(r)}$ forms an edge if 
$$\Gamma\left(f_1(x)\leq 0,\dots,f_m(x)\leq 0\right)=\texttt{true}.$$
Then $H$, and any hypergraph isomorphic to $H$, is a \textbf{semialgebraic hypergraph of description complexity} $(d,D,m)$. 
\end{definition}

We remark that several closely related definitions of semialgebraic hypergraphs appear in the literature. For instance, Conlon, Fox, Pach, Sudakov, and Suk \cite{CFPSS14} define complexity as a pair $(d,t)$, where $t$ serves as a simultaneous upper bound for both $m$ and $D$.  In contrast,  Tidor and Yu \cite{TY} defines complexity $(d,D)$, where $D$ is the sum of the degrees of the $m$ polynomials $f_1,\dots,f_m$. Moreover, they work with unordered vertex sets $V$, but in exchange the edge relation has to be assumed symmetric.   Our  choice of $(d, D, m)$ provides a finer description of the complexity, as we are mainly interested in the degrees $D$. However, our definition does not change the core concept of semialgebraic relations.

\subsection{Exponential sequences}

A core component of the proof of our main theorem is the study of exponential subsequences, which we discuss in detail in this section.

\begin{definition}[Exponential sequence]
A sequence of real numbers  $\{a_i\}_{i=1,\dots,n}$ is 
\begin{itemize}
    \item  \textbf{exponential} if $0\leq  2a_{i-1}\leq a_i$ for $i=2,\dots,n$;
    \item  \textbf{weakly-exponential} of type
    \begin{description}
        \item[\hspace{10pt} $\typeone$] if $\{a_i\}_{i=1,\dots,n}$ is exponential,
        \item[\hspace{10pt} $\typetwo$] if $\{a_i\}_{i=n,\dots,1}$ is exponential,
        \item[\hspace{10pt} $\typethree$] if $\{-a_i\}_{i=1,\dots,n}$ is exponential,
        \item[\hspace{10pt} $\typefour$] if $\{-a_i\}_{i=n,\dots,1}$ is exponential;
     \end{description}
    \item  \textbf{shifted-exponential} of type $\tau\in \{\typeone,\typetwo,\typethree,\typefour\}$ if there exists $t\in \mathbb{R}$ such that $\{a_i-t\}_{i=1,\dots,n}$ is weakly-exponential of type $\tau$.
\end{itemize}
\end{definition}

The strength of shifted-exponential sequences comes from the fact that any sequence of $n$ real numbers contains a shifted-exponential subsequence of length $\Omega(\log n)$, and this bound is tight \cite{JinTomon}. This can be viewed as an analogue of the celebrated Erd\H{o}s-Szekeres theorem \cite{ESz35}, which ensures that any sequence of length $n$ contains a monotone subsequence of length $\lceil \sqrt{n}\rceil.$ An important idea, which also played a key role in the work of Jin and Tomon \cite{JinTomon}, is that shifted-exponential subsequences can be found with the help of coloring triplets. This motivates the following definition.  

\begin{definition}
    A triple of real numbers $(a,b,c)$ has type
    \begin{description}
        \item[\hspace{10pt} $\typeone$] if $0\leq b-a\leq c-b$,
        \item[\hspace{10pt} $\typetwo$] if $b-a\geq c-b\geq 0$,
        \item[\hspace{10pt} $\typethree$] if $0\leq a-b\leq b-c$,
        \item[\hspace{10pt} $\typefour$] if $a-b\geq b-c\geq 0$,
        \item[\hspace{10pt} $\typezero$] if $a,b,c$ is not monotone.
    \end{description}
\end{definition}

\noindent
In case a triple or sequence has more than one type, we assign a type arbitrarily.

\begin{lemma}\label{lemma:cups_caps}
Let $\{a_i\}_{i=1,\dots,n}$ be a sequence of real numbers. If every triple $(a_i,a_j,a_k)$ for $1\leq i<j<k\leq n$ has type $\tau\in\{\typeone,\typetwo,\typethree,\typefour\}$, then $\{a_i\}_{i=1,\dots,n}$ is shifted-exponential of type $\tau$. Moreover, if every triple has type $\typezero$, then $n\leq 4$.
\end{lemma}

\begin{proof}
 It is an easy exercise to show that any sequence of 5 numbers contains a subsequence of size 3 that is monotone, so a sequence can only be type $\typezero$ if it has at most 4 elements.

 Now assume that $\{a_i\}_{i=1,\dots,n}$ has a non-$\typezero$ type. We only consider type $\typeone$, the other three cases can be handled similarly. We show that $\{a_{i}-a_1\}_{i=1,\dots,n}$ is exponential, which then implies that $\{a_i\}_{i=1,\dots,n}$ is shifted-exponential of type $\typeone$. Indeed, $(a_1,a_i,a_{i+1})$ is of type $\typeone$, which means that $0\leq a_i-a_1\leq a_{i+1}-a_{i}$, which then implies $0\leq 2(a_i-a_1)\leq a_{i+1}-a_1$.  
\end{proof}

In general, given a function  $\phi:A\rightarrow\mathbb{R}$ (or sequence of real numbers), a \emph{shifting} of this function refers to any function of the form  $\phi-t$ for some $t\in \mathbb{R}$. We also define a relaxation of exponential sequences to general real functions. This notion will be used in the case of functions acting on $D$-tuples, i.e. $\phi:A^{(D)}\rightarrow\mathbb{R}$.

\begin{definition}
    A function $\phi:B\rightarrow \mathbb{R}$ is \textbf{exponentially separated} if $\phi(b)$ has the same sign for every $b\in B$, and for any two distinct $b\neq b'$, we have 
    $$|\phi(b)|\geq 2|\phi(b')|\mbox{ or }|\phi(b')|\geq 2|\phi(b)|.$$
    We define exponentially separated sequences analogously.
\end{definition}

Equivalently, $\phi:B\rightarrow \mathbb{R}$ is exponentially separated if and only if there is an enumeration $b_1,\dots,b_n$ of the elements of $B$ such that $\{\phi(b_i)\}_{i=1,\dots,n}$ is exponential (or weakly-exponential).

\subsection{Extremal combinatorics}

We present some simple or well known auxiliary results about hypergraphs. First, we state the quantitative version of Ramsey's theorem for colorings with constant number of colors, proved originally by Erd\H{o}s and Rado \cite{ER52}.

\begin{theorem}\label{thm:Ramsey}
    Let $r,s$ be positive integer constants. Let $H$ be the complete $r$-uniform hypergraph on $N$ vertices, whose edges are colored with $s$ colors. Then $H$ contains a monochromatic clique of size $n$, where $N\leq \tw_{r}(O(n))$.
\end{theorem}

We also make use of the following lemma, which can be thought of as an extension of the so called anti-Ramsey theorem of Babai \cite{Babai}, see also \cite{AJMP}.

\begin{lemma}\label{lemma:rainbow}
Let $r,s,q$ be positive integer constants. Let $H$ be the complete $r$-uniform hypergraph on $N$ vertices, and let $\gamma:E(H)\rightarrow \mathbb{R}^s$ satisfy the following property. For every edge $f$, there are at most $q$ edges $f'$ such that $|f\cap f'|=r-1$ and $|\gamma(f)_i-\gamma(f')_i|<1$ for some $i\in [s]$. Then there exists a complete subhypergraph $H'$ on $\Omega(N^{1/(2r-1)})$ vertices such that $|\gamma(f)_i-\gamma(f')_i|\geq 1$ for every distinct pair of edges $f,f'\in E(H')$ and $i\in [s]$. 
\end{lemma}

\begin{proof}
 Let $$S=\{(f,f')\in E(H): f\neq f',\exists i\in [s], |\gamma(f)_i-\gamma(f')_i|<1\},$$
so that $S$ is the set of ''conflicts''. Fix $f\in E(H)$, then 
$$\#\{f': (f,f')\in S,  |f\cap f'|=k\}\leq 2^rqs  N^{r-k-1}.$$
Indeed, for each $g\subset f$ of size $k$, the number of $f'$ such that $(f,f')\in S$ and $g=f\cap f'$ is at most $qsN^{r-k-1}$. This follows from the observation that the number of $(r-1)$-tuples $g$ is contained in is at most $N^{r-k-1}$, and each $(r-1)$-tuple can be contained in at most $qs$ such edges $f'$. 

 The rest of the proof is a standard application of the probabilistic deletion method. Let $p=cN^{-\frac{2r-2}{2r-1}}$ with $c=1/(2^{r+1}rqs)$, and sample the vertices of $H$ with probability $p$, with $U$ denoting the set of sampled vertices.   Let $X$ be the number of pairs $(f,f')\in S$ such that $f,f'\subset U$. Then
\begin{align*}
\mathbb{E}(X)&=\sum_{(f,f')\in S} p^{|f\cup f'|}=\sum_{k=0}^{r-1}p^{2r-k}\cdot \#\{(f,f')\in S: |f\cap f'|=k\}\\
&\leq \sum_{k=0}^{r-1}p^{2r-k}N^{r}\cdot (2^rqs N^{r-k-1})= 2^rqs\sum_{k=0}^{r-1}c^{2r-k}N^{\frac{2r-k}{2r-1}-1}<2^r rqsc^{r+1}N^{\frac{1}{2r-1}}<\frac{c}{2}N^{\frac{1}{2r-1}}.
\end{align*}
 For each $(f,f')\in S$ such that $f,f'\subset U$, delete a vertex of $f\cup f'$ from $U$, and let $U'$ be the resulting set. Then $\mathbb{E}(|U'|)\geq \mathbb{E}(|U|-X)\geq cN^{1/(2r-1)}/2$. Fix an outcome of the sampling such that $|U'|\geq cN^{1/(2r-1)}/2$, then $H[U']$ satisfies the required properties.
\end{proof}

\section{Proof outline}\label{sect:outline}

We give a rough outline of the proof of \Cref{thm:main}. First, we discuss the proof of Jin and Tomon \cite{JinTomon} of the $D=1$ case, then we present our new key ideas by considering the case $D=2$. For further simplification, we assume that the number of defining polynomials is $m=1$, by noting that the analysis of the general case does not increase the complexity significantly.

\medskip

First, let $H$ be an $N$-vertex $r$-uniform semialgebraic hypergraph of description complexity $(d,1,1)$. Then there exists a linear function $f:(\mathbb{R}^d)^r\rightarrow\mathbb{R}$ such that $x\in V(H)^{(r)}$ is an edge if and only if $f(x)\leq 0$. Using that $f$ is linear, we can write $f(x)=\sum_{i=1}^{r}\phi_i(x_i)$
with suitable functions $\phi_i:\mathbb{R}^d\rightarrow\mathbb{R}$, $i\in [r]$. The key idea is to consider the sequences $\{\phi_i(x_i)\}_{x\in V(H)}$ for $i\in [r]$, and find a set $S\subset V(H)$  of size $(\log N)^{\Omega(1)}$ such that all $r$ sequences $\{\phi_i(x_i)\}_{x\in S}$ are shifted-exponential. After shifting, we get functions $\sigma_i:S\rightarrow \mathbb{R}$ and a constant $C$ such that $f(x)=C+\sum_{i=1}^r\sigma_i(x_i)$, and each sequence $\{\sigma_i(x_i)\}_{x\in S}$ is weakly-exponential. The advantage of this is that for a ''generic'' $r$-tuple $x\in S^{(r)}$, the sum $C+\sum_{i=1}^r\sigma_i(x_i)$ is dominated by the term with the largest absolute value, so its sign is determined by the sign of this term. To finish the proof, we find $T\subset S$ of size $|S|^{\Omega(1)}$ such that every $x\in T^{(r)}$ is generic, and the index of the largest element of $(C,|\sigma_1(x_1)|,\dots,|\sigma_r(x_r)|)$ does not depend on the choice of $x$. This ensures that $f(x)$ has the same sign for every $x\in T^{(r)}$, so $T$ is either a clique or an independent set of size $(\log N)^{\Omega(1)}$ in $H$.

\medskip

Next, let $H$ be an $N$-vertex $r$-uniform semialgebraic hypergraph of description complexity $(d,2,1)$. Then there exists $f:(\mathbb{R}^d)^r\rightarrow\mathbb{R}$ of degree at most 2 such that $x\in V(H)^{(r)}$ is an edge if and only if $f(x)\leq 0$. Using that $f$ has degree at most 2, we can write $f(x)=\sum_{I\in [r]^{(2)}}\phi_I(x_I)$
with suitable functions $\phi_I:(\mathbb{R}^d)^2\rightarrow\mathbb{R}$, $I\in [r]^{(2)}$. Now $\phi_I$ acts on the pair of vertices of $V(H)$, so $\{\phi_I(x_I)\}_{x\in V(H)}$ corresponds to an edge weighting of a complete graph instead of a sequence. Therefore, if one wants to adapt our approach for the $D=1$ case, one needs an analogue of shifted-exponential sequences for graph weightings. Our ultimate goal is to find a large subset $S$ such that $\sigma_I:S^{(2)}\rightarrow \mathbb{R}$ is exponentially separated for every $I\subset [r]^{(2)}$, where $\sigma_I$ is a slightly modified version of $\phi_I$. Unfortunately, allowing only shiftings will not be sufficient to achieve this goal.

Indeed, consider the following example. Let $\phi:[n]^{(2)}\rightarrow \mathbb{R}$ be defined as $\phi(a,b)=n^{a}+b$. It is easy to see that no shifting of $\phi$ is exponentially separated on more than 3 elements of $[n]$. The issue is that $\phi(a,b)\approx \rho(a)$ for some function $\rho$ only depending on $a$. To overcome this, we allow functions $\sigma$ of the form $\sigma(a,b)=\phi(a,b)-\rho_1(a)-\rho_2(b)$, and write $\phi\sim \sigma$ for all such functions $\sigma$. This resolves the issue: we show that there exist $S\subset V(H)$ and $\sigma_I:S^{(2)}\rightarrow\mathbb{R}$, $I\in [r]^{(2)}$, such that  $N\leq \tw_{8}(O(|S|^3))$, $\phi_I\sim \sigma_I$ and $\sigma_I$ is exponentially separated. Now, we can rewrite
$$f(x)=\sum_{I\in [r]^{(2)}}\sigma_I(x_I)+\sum_{i\in [r]}\phi_i(x_i)$$
with suitable functions $\phi_i:S\rightarrow \mathbb{R}$. 

We further pass to a subset $T\subset S$ of size $(\log |S|)^{\Omega(1)}$ such that $\{\phi_i(x_i)\}_{x\in T}$ is shifted-exponential for every $i\in [r]$, and rewrite
\begin{equation}\label{eq:1}
f(x)=C+\sum_{I\in [r]^{(2)}}\sigma_I(x_I)+\sum_{i\in [r]}\sigma_i(x_i),
\end{equation}
with suitable $C\in \mathbb{R}$ and $\sigma_i$, $i\in [r]$, each of which is weakly-exponential on $T$. The rest of the proof is very similar as in the $D=1$ case. We observe that for a ''generic'' $x\in T^{(r)}$, the sum in the right-hand-side of (\ref{eq:1}) is dominated by one of the terms, so the sign of $f$ is the same as the sign of this dominant term. To finish the proof, we find $U\subset T$ such that $|T|\leq \tw_{4}(O(|U|))$, every $x\in U^{(r)}$ is generic, and the index of the largest element of $(C,(|\sigma_i(x_i)|)_{i\in [r]},(\sigma_{I}(x_I))_{I\in [r]^{(2)}})$ does not depend on the choice of $x$. This ensures that $f(x)$ has the same sign for every $x\in U^{(r)}$, so $U$ is either a clique or an independent set in $H$.

The most difficult part of the proof is finding the desired set $S$ and functions $\sigma_I$. For simplicity, consider a single function $\phi:A^{(2)}\rightarrow \mathbb{R}$ on some finite ordered set $A$. We define the evaluation of $\phi$ as a function $\psi_{\phi}$ acting on the 4-tuples of $A$ such that for $(a,b,c,d)\in A^{(4)}$, $$\psi_{\phi}(a,b,c,d)=\phi(a,c)-\phi(a,d)-\phi(b,c)+\phi(b,d).$$
Here, $\psi_{\phi}$ has the advantage that $\psi_{\phi}\equiv \psi_{\sigma}$ for every $\phi\sim \sigma$. In two rounds, we define certain colorings on the 5- and then 4-tuples of $S$ with the help of $\psi_{\phi}$, and use Ramsey's theorem to find a large monochromatic set $B$. This set $B$ allows the construction of some $\sigma\sim \phi$ that is exponentially separated on~$B$.

\section{Warm-up: degree 2 semialgebraic hypergraphs}\label{sect:D=2}

In this section, we prove the $(D,m)=(2,1)$ subcase of \Cref{thm:main} in a somewhat simplified manner. This serves to illustrate the key ideas of the general proof while sidestepping the more technical details.

Given an ordered set $A$, define an equivalence relation on the space of functions $\{\phi:A^{(2)}\rightarrow \mathbb{R}\}$ as follows. Write $\phi\sim \phi'$ if there exist functions $\rho_1,\rho_2:A\rightarrow\mathbb{R}$ such that $$\phi(a,b)=\phi'(a,b)-\rho_1(a)-\rho_2(b)$$ for every $(a,b)\in A^{(2)}.$
Given a function $\phi:A^{(2)}\rightarrow \mathbb{R}$, we define the evaluation $\psi_{\phi}:A^{(4)}\rightarrow\mathbb{R}$ such that
$$\psi_{\phi}(a,b,c,d)=\phi(a,c)-\phi(a,d)-\phi(b,c)+\phi(b,d).$$
The advantage of $\psi_{\phi}$ is that it is invariant on the equivalence class of $\phi$, that is, if $\phi\sim \phi'$, then $\psi_{\phi}\equiv \psi_{\phi'}$. If $\phi$ is clear from the context, we write simply $\psi$ instead of $\psi_{\phi}$. The following is our main technical lemma.

\begin{lemma}\label{lemma:exp_sep}
Let $\phi:[N]^{(2)}\rightarrow \mathbb{R}$. Then there exist $S\subset [N]$ of size $n$  for some $N\leq \tw_8(O(n^{3}))$ and $\sigma:S^{(2)}\rightarrow \mathbb{R}$ such that $\sigma\sim \phi$ on $S$ and $\sigma$ is exponentially separated.
\end{lemma}

\begin{proof}
 Define a coloring of $[N]^{(5)}$ with a color from $\{\typezero,\dots,\typefour\}$. We color $(a,b,c,d,e)$ depending on the relation of $X=\psi(a,b,d,e)$ and $Y=\psi(b,c,d,e)$, where $\psi=\psi_{\phi}$. We use color
\begin{description}
    \item[\hspace{10pt} $\typeone$] if $0\leq X\leq Y$,
    \item[\hspace{10pt} $\typetwo$] if $0\leq Y\leq X$,
    \item[\hspace{10pt} $\typethree$] if $0\leq (-X)\leq (-Y)$,
    \item[\hspace{10pt} $\typefour$] if $0\leq (-Y)\leq (-X)$,
    \item[\hspace{10pt} $\typezero$] if $X$ and $Y$ has different signs.
\end{description}
 Using the qualitative form of Ramsey's theorem (\Cref{thm:Ramsey}), there exists a monochromatic $S_1\subset [N]$ such that $N\leq \tw_5(O(|S_1|))$. Let $\tau_1$ denote the color of $S_1$. Assuming $N$ is sufficiently large, we have $\tau_1\neq \typezero$, as any set completely colored with $\typezero$ has size at most $6$.

For ease of notation, relabel the elements of $S_1$ such that $S_1=[N_1]$. Fix some $d\leq N_1-1$, and consider the sequence 
$$x^{(d)}_a=x_a=\phi(a,d+1)-\phi(a,d)$$
for $a=1,\dots,d-1$. We have $\psi(a,b,d,d+1)=x_b-x_a$, so the fact that $S_1$ is monochromatic of color $\tau_1$ implies that $\{x_a\}_{a=1,\dots,d-1}$ is a shifted-exponential sequence of type $\tau_1$ by \Cref{lemma:cups_caps}. Therefore, there exists $t_d$ such that $\{x^{(d)}_a-t_d\}_{a=1,\dots,d-1}$ is weakly-exponential of type $\tau_1$. Let $u_d=\sum_{i<d}t_i$ and define $\phi_1:S_1^{(2)}\rightarrow \mathbb{R}$ as $\phi_1(a,d)=\phi(a,d)-u_d$, then $\phi_1\sim \phi$ and $x_a-t_d=\phi_1(a,d+1)-\phi_1(a,d)$. Hence, we get that $$\{\phi_1(a,d+1)-\phi_1(a,d)\}_{a=1,\dots,d-1}$$ is weakly-exponential of type $\tau_1$. Now observe that for $a<d<e$, we have
\begin{align*}
    \phi_1(a,e)-\phi_1(a,d)&=\phi(a,e)-\phi(a,d)-(u_e-u_d)\\
    &=\sum_{d_0=d}^{e-1}\phi(a,d_0+1)-\phi(a,d_0)-t_{d_0}\\
    &=\sum_{d_0=d}^{e-1}\phi_1(a,d_0+1)-\phi_1(a,d_0).
\end{align*}
As the sum of weakly-exponential sequences of the same type is also weakly-exponential, we get that for every $d<e\leq N_1$, the sequence
$$\{\phi_1(a,e)-\phi_1(a,d)\}_{a=1,\dots,d-1}$$
is weakly-exponential of type $\tau_1$.

Next, we define a coloring of $S_1^{(4)}$ with 5 colors. Given $I=(a,b,c,d)\in S_1^{(4)}$, color $I$ with the type of the triple $(\phi_1(a,b),\phi_1(a,c),\phi_1(a,d))$. By \Cref{thm:Ramsey}, there exists a monochromatic $S_2\subset S_1$ such that $|S_1|\leq \tw_4(O(|S_2|))$. Also, if $|S_2|>5$, the color of $S_2$ is not type $\typezero$. To simplify notation, relabel the elements of $S_2$ such that $S_2=[N_2]$. Then $S_2$ being monochromatic implies that for every $a\in [N_2]$, the sequence $\{\phi_1(a,b)\}_{b=a+1,\dots,N_2}$ is shifted-exponential of type $\tau_2$ for some $\tau_2\in [4]$. Therefore, there exists $w_a$ such that $\{\phi_1(a,b)-w_a\}_{b=a+1,\dots,N_2}$ is weakly-exponential of type $\tau_2$. Define $\sigma:S_2^{(2)}\rightarrow \mathbb{R}$ such that 
$$\sigma(a,b)=\phi_1(a,b)-w_a,$$
and observe that $\sigma(a,e)-\sigma(a,d)=\phi_1(a,e)-\phi_1(a,d)$ for every $a<d<e$. For convenience, we collect the important properties of $\sigma$:
\begin{itemize}
    \item $\sigma\sim \phi_1\sim \phi$ on $S_2$,
    \item $\{\sigma(a,b)\}_{b=a+1,\dots,N_2}$ is weakly-exponential of type $\tau_2$ for every $a\in S_2$,
    \item $\{\sigma(a,e)-\sigma(a,d)\}_{a=1,\dots,d-1}$ is weakly-exponential of type $\tau_1$ for every $d<e\leq N_2$.
\end{itemize}
For the sake of simplicity, we assume that $(\tau_1,\tau_2)=(\typeone,\typeone)$. The other cases can be handled similarly, which we will demonstrate in the detailed proof of the more general \Cref{lemma:exp_sep_general}. Let $S_3$ be the subset of even numbers in $S_2=[N_2]$. We show that $S_3$ satisfies the following property: 
\begin{enumerate}
    \item[4.] $\{\sigma(a,b)\}_{a\in S_3:a<b}$ is exponential for every $b\in S_3$.
\end{enumerate}
It is enough to show that for every $(a,b)\in S_3^{(2)}$ such that $a+2<b$, we have $0\leq 2\sigma(a,b)\leq \sigma(a+2,b)$. Note that $(a,a+2,b-1,b)\in S_2^{(4)}$, hence, using 2. and 3.,
\begin{itemize}
    \item[(i)]  $0\leq 2\sigma(a,b-1)\leq \sigma(a,b)$,
    \item[(ii)] $0\leq 2\sigma(a+2,b-1)\leq \sigma(a+2,b)$,
    \item[(iii)] $0\leq 4(\sigma(a,b)-\sigma(a,b-1))\leq \sigma(a+2,b)-\sigma(a+2,b-1).$
\end{itemize}
Combining these inequalities, we get
$$0\leq 2\sigma(a,b)\leq  4(\sigma(a,b)-\sigma(a,b-1))\leq \sigma(a+2,b)-\sigma(a+2,b-1)\leq \sigma(a+2,b),$$
where the first and second inequality holds by (i), the third inequality holds by (iii), and the last inequality holds by (ii). From this, we see that $0\leq 2\sigma(a,b)\leq \sigma(a+2,b),$ confirming the claim.

 Finally, we apply \Cref{lemma:rainbow} to find a subset $S\subset S_3$ of size $\Omega(|S_3|^{1/3})$ so that $\sigma$ is exponentially separated on $S$. Indeed, define the function $\gamma:S_3^{(2)}\rightarrow \mathbb{R}$ such that $\gamma(a,b)=\log_2 \sigma(a,b)$. Then by properties 3. and 4., $\gamma$ satisfies that for every $(a,b)\in S_3^{(2)}$, there are at most 4 other $(c,d)\in S_3^{(2)}$ such that $|\gamma(a,b)-\gamma(c,d)|< 1$, and $|\{a,b\}\cap\{c,d\}|=1$. Hence, \Cref{lemma:rainbow} ensures the existence of a set $S\subset S_3$ of size $\Omega(|S_3|^{1/3})$ such that $|\gamma(a,b)-\gamma(c,d)|\geq 1$ for every choice of distinct $(a,b),(c,d)$. Equivalently, $\sigma$ is exponentially separated on $S$.

 We finish the proof by bounding the size of $S$. We have $N\leq \tw_5(O(|S_1|))$, $|S_1|\leq \tw_4(O(|S_2|))$, $|S_3|=O(|S_2|)$, and $|S|=O(|S_3|^3)$. From this,  $N=\tw_{8}(O(n^{3}))$, so the set $S$ and function $\sigma$ satisfy the desired conditions.
\end{proof}

We need a strengthening of this lemma that applies to a small set $\Phi$ of functions $\phi:A^{(2)}\rightarrow \mathbb{R}$ simultaneously. The proof of this is essentially identical to the previous proof, but slightly more technical. We omit its proof, as a more general statement is proved later as \Cref{lemma:exp_sep_general}.

\begin{lemma}\label{lemma:exp_sep2}
Let $s$ be a positive integer constant, and let $\phi_1,\dots,\phi_s:[N]^{(2)}\rightarrow \mathbb{R}$. Then there exist $S\subset [N]$ of size $n$  for some $N\leq \tw_8(O(n^{3}))$ and $\phi_1',\dots,\phi_s':S^{(2)}\rightarrow \mathbb{R}$ such that $\sigma_i\sim \phi_i$ on $S$ and $\sigma_i$ is exponentially separated for $i\in [s]$.
\end{lemma}

Now we are ready to present the proof of the main theorem in the special case $(d,D,m)=(d,2,1)$.

\begin{theorem}
Let $r$ be a positive integer constant. Let $H$ be an $r$-uniform semialgebraic hypergraph of description complexity $(d,2,1)$ on $N$ vertices. Then $H$ contains a clique or an independent set of size $n$, where $N\leq \tw_{12}(O(n))$.
\end{theorem}

\begin{proof}
Let $f:(\mathbb{R}^d)^r\rightarrow \mathbb{R}$ be a polynomial of degree at most $2$ such that $x\in V(H)^{(r)}$ is an edge if $f(x)\leq 0$. Observe that we can write
$$f(x)=\sum_{I\in [r]^{(2)}}\phi_I(x_I),$$
where $\phi_I:(\mathbb{R}^d)^2\rightarrow \mathbb{R}$ is a polynomial of degree at most 2 for every $I\in [r]^{(2)}$. We note that the choice of $(\phi_I)_{I\in [r]^{(2)}}$ is not unique, but we fix any choice. In what follows, we will not use that $\phi_I$ is a polynomial, we only care that $\phi_I$ acts on $V(H)^{(2)}$. 

Using \Cref{lemma:exp_sep2}, there exist $S_1\subset V(H)$  and $\sigma_I:S_1^{(2)}\rightarrow \mathbb{R}$ such that $N\leq \tw_{8}(O(|S_1|^{3}))$, $\sigma_I\sim \phi_I$ on $S_1$, and $\sigma_I$ is exponentially separated for $I\in [r]^{(2)}$. This means that there exist $\rho_{I,1},\rho_{I,2}:S_1\rightarrow \mathbb{R}$ such that $$\phi_I(a,b)=\sigma(a,b)+\rho_{I,1}(a)+\rho_{I,2}(b).$$
For $i\in [r]$ and $a\in S_1$, define $$\phi_i(a)=\sum_{(i,j)\in [r]^{(2)}}\rho_{(i,j),1}(a)+\sum_{(j,i)\in [r]^{(2)}}\rho_{(j,i),2}(a),$$ 
then
$$f(x)=\sum_{I\in [r]^{(2)}}\phi_I(x_I)=\sum_{I\in [r]^{(2)}}\sigma_I(x_I)+\sum_{i\in [r]}\phi_i(x_i).$$

Next, we cite Lemma 12 in \cite{JinTomon} to find a subset $S\subset S_1$ and numbers $w_i$ such that setting $\sigma_i(a)=\phi_i(a)-w_i$, the sequence $\{\sigma_i(a)\}_{a\in S}$ is weakly-exponential, and $|S_1|\leq 2^{|S|^{O(1)}}$. Let $\sigma_{\emptyset}=\sum_{i\in [r]}w_i$, then
$$\sum_{I\in [r]^{(2)}}\phi_I(x_I)=\sum_{I\in [r]^{(\leq 2)}}\sigma_I(x_I).$$

Next, define a coloring of the complete 4-uniform hypergraph on $S$ as follows. Given a 4-tuple $a\in S^{(4)}$, consider all the numbers $|\sigma_I(a_J)|$ for $I\in [r]^{(\leq 2)}$ and $J\subset [4]^{(|I|)}$. Order these $R=6\binom{r}{2}+4r+1$ numbers in increasing order, and record the \emph{signatures} of the numbers. Here, the signature of $|\sigma_I(a_J)|$ is $(I,J)$. Also, say that $a$ is \emph{bad} if 
$$\frac{1}{10r^2}|\sigma_I(a_J)|<|\sigma_{I'}(a_{J'})|<10r^2|\sigma_I(a_J)|$$ for some $(I,J)$ and $(I',J')$ with $J\neq J'$. Color $a$ with color 0 if it is bad, otherwise color it with the sequence of $R$ signatures. This defines a coloring of $S^{(4)}$ with at most $R!+1=O(1)$ numbers. Applying \Cref{thm:Ramsey}, there exists monochromatic $T\subset S$ such that $|S|\leq \tw_4(O(|T|)$. 

First, observe that $T$ cannot be color 0, assuming $N$ is sufficiently large. Indeed, using that $\sigma_I$ is exponentially separated for every $I\in [r]^{(\leq 2)}$, it is easy to argue that a clique of color 0 has size $O(1)$. Therefore, $T$ is a monochromatic clique with no bad 4-tuples. Consider $x\in T^{(r)}$, which is an edge of $H$ if and only if  
$$\sum_{I\in [r]^{(\leq 2)}}\sigma_I(x_I)\leq 0.$$
The sum on the left hand side contains $\binom{r}{2}+r+1$ numbers, and if $z$ and $z'$ are two such numbers, then $|z|/|z'|\not\in (1/10r^2,10r^2)$, using the property the no element of $T^{(4)}$ is bad. But this means that the sign of the sum is the sign of the number with the largest absolute value. Using that $T$ is monochromatic, we get that there is $I_0\subset [r]^{(\leq 2)}$ such that $|\sigma_{I_0}(x_{I_0})|$ is always the largest, as the order of the numbers $\{|\sigma_{I}(x_I)|:I\in [r]^{(\leq 2)}\}$ is uniquely determined by the color of $T$. But as $\sigma_{I_0}(x_{I_0})$ has the same sign for all choices $x\in T^{(r)}$, we conclude that $T$ is either a clique or an independent set in $H$. 

We finish the proof by bounding size of $S$. We have $N\leq \tw_{8}(O(|S_1|^{3}))$, $|S_1|\leq 2^{|S|^{O(1)}}$, and $|S|\leq \tw_4(O(|T|))$. Therefore, $N\leq \tw_{12}(O(|T|))$.
\end{proof}

\section{Low-degree semialgebraic hypergraphs}\label{sect:proof}

In this section, we give a detailed proof of \Cref{thm:main}. The proof is divided into three subsections. In the first section, we discuss properties of functions of the form $\phi:A^{(D)}\rightarrow\mathbb{R}$, and define the natural extension of the evaluation $\psi$. In the second subsection, we prove our main technical lemma, which generalizes  \Cref{lemma:exp_sep} for $D$-tuples. We finish the proof in the third subsection.

\subsection{Functions defined on tuples}

Given an ordered set $A$, define an equivalence relation on the space of functions $\{\phi:A^{(D)}\rightarrow \mathbb{R}\}$ as follows. Write $\phi\sim \phi'$ if there exist $D$ functions $\rho_J:A^{(D-1)}\rightarrow \mathbb{R}$ for $J\in [D]^{(D-1)}$ such that for  every $a\in A^{(D)}$,
$$\phi(a)=\phi'(a)-\sum_{J\in [D]^{(D-1)}}\rho_J(a_J).$$
Given a function $\phi:A^{(D)}\rightarrow \mathbb{R}$,  we define $\psi_{\phi}:A^{(2D)}\rightarrow\mathbb{R}$ as
$$\psi_{\phi}(a)=\sum_{\eps\in \{0,1\}^{D}}(-1)^{\eps_1+\dots+\eps_{D}}\phi((a_{2i-\eps_i})_{i=1,\dots,D}).$$
We refer to $\psi_{\phi}$ as the \emph{hypercubic-evaluation} of $\phi$. For example, in case $D=1$, we have $\psi_{\phi}(a)=\phi(a_2)-\phi(a_1)$, while the $D=2$ case coincides with the definition in the previous section. In general, we may view the input as $D$ pairs $(a_1,a_2)$, $(a_3,a_4)$, $\dots$, $(a_{2D-1},a_{2D})$, and then we sum (with appropriate signs) the values of $\phi$ on the $D$-dimensional hypercube $\{a_1,a_2\}\times \{a_3,a_4\}\times\dots\times \{a_{2D-1},a_{2D}\}$.

We collect some simple properties of $\psi$. We say that $\phi:A^{(D)}\rightarrow \mathbb{R}$ \emph{does not depend on coordinate $i$} if $\phi(a)=\phi(a')$ for every $a,a'\in A^{(D)}$ where $a$ and $a'$ only differ in the $i$-th coordinate.
\begin{enumerate}[label={(\arabic*)}]
\item\label{property:1} Given $(a_1,a_2,b)\in A^{(2D)}$ with $a_1,a_2\in A$ and $b\in A^{(2D-2)}$, we have
$$\psi_{\phi}(a_1,a_2,b)=\psi_{\phi(a_2,.)}(b)-\psi_{\phi(a_1,.)}(b).$$
\item\label{property:2} If $\phi$ does not depend on one of the coordinates, then $\psi_{\phi}\equiv 0$.
\item\label{property:3} For every $\phi_1,\phi_2:A^{(D)}\rightarrow \mathbb{R}$,
$$\psi_{\phi_1}+\psi_{\phi_2}\equiv \psi_{\phi_1+\phi_2}.$$
\item\label{property:4} Given $(a,b_1,b_2,b_3,c)\in A^{(2D+1)}$ such that $a\in A^{(2k)}$ and $b_1,b_2,b_3\in A$ and $c\in A^{(2D-2k+1)}$, we have
$$\psi_{\phi}(a,b_1,b_2,c)+\psi_{\phi}(a,b_2,b_3,c)=\psi_{\phi}(a,b_1,b_3,c).$$
\item\label{property:5} $\psi_{\phi}$ is invariant on the equivalence class of $\phi$, that is, if $\phi\sim\phi'$, then $\psi_{\phi}\equiv \psi_{\phi'}$.
\item\label{property:6}  Assume that $A=[N]$. For every $a\in A^{(2D)}$,
$$\psi_{\phi}(a)=\sum_{b_1=a_1}^{a_2-1}\sum_{b_2=a_3}^{a_4-1}\dots\sum_{b_D=a_{2D-1}}^{a_{2D}-1}\psi_{\phi}(b_1,b_1+1,b_2,b_2+1,\dots,b_D,b_D+1).$$
\end{enumerate}
Indeed, \ref{property:1}-\ref{property:4} easily follow from the definition, \ref{property:5} follows from the combination of 2. and 3., and \ref{property:6} follows from repeated applications of \ref{property:4}. The following technical lemma is used later.

\begin{lemma}\label{lemma:phi_construction}
Let $\rho:A^{(D)}\rightarrow \mathbb{R}$ and define the function $\pi:A^{(2D)}\rightarrow  \mathbb{R}$ such that for every $a\in A^{(2D)}$,
$$\pi(a)=\sum_{a_1\leq \alpha_1<a_2}\sum_{a_3\leq \alpha_2<a_4}\dots\sum_{a_{2D-1}\leq \alpha_D<a_{2D}}\rho(\alpha_1,\dots,\alpha_D).$$
Then there exists $\phi:A^{(D)}\rightarrow\mathbb{R}$ such that $\psi_{\phi}\equiv \pi$.
\end{lemma}

\begin{proof}
Define $\phi$ such that for every $b\in A^{(D)}$,
$$\phi(b)=\sum_{\beta_1<b_1}\sum_{\beta_2<a_2}\dots\sum_{\beta_D<b_D}\rho(\beta_1,\dots,\beta_D).$$
We show that $\phi$ suffices by induction on $D$. In case $D=1$, we have
$$ \pi(a_1,a_2)=\sum_{a_1\leq \alpha<a_2}\rho(\alpha)=\sum_{\alpha<a_2}\rho(\alpha)-\sum_{\alpha<a_1}\rho(\alpha)=\phi(a_2)-\phi(a_1)=\psi_{\phi}(a_1,a_2),$$
so the claim holds. Now assume that $D\geq 2$. Without loss of generality, assume that $A=[N]$. For $t\in [N]$, define $ \pi_t:[t+2,N]^{(2D-2)}$ as follows. Let $a\in [t+2,N]^{(2D-2)}$, then set 
$$\pi_t(a):= \pi(t,t+1,a)=\sum_{a_1\leq \alpha_1<a_2}\dots\sum_{a_{2D-3}\leq \alpha_{D-1}<a_{2D-2}}\rho(t,\alpha_1,\dots,\alpha_{D-1}).$$
Then for every $(t,u,a)\in [N]^{(2D)}$ with $t,u\in [N]$, we have
$$ \pi(t,u,a)=\sum_{t\leq \alpha<u} \pi_{\alpha}(a).$$
Consider the function $\rho_t:[t+2,N]^{(D-1)}\rightarrow\mathbb{R}$ defined as $\rho_t(b):=\rho(t,b)$ for $b\in [t+2,N]^{(D-1)}$. Then for $a\in [t+2,N]^{(2D-2)}$,
$$ \pi_t(a)=\sum_{a_1\leq \alpha_1<a_2}\dots\sum_{a_{2D-3}\leq \alpha_{D-1}<a_{2D-2}}\rho_t(\alpha_1,\dots,\alpha_{D-1}).$$
Therefore, we can apply our induction hypothesis to deduce that the function $\phi_t:[t+2,N]^{(D-1)}$ defined as 
$$\phi_t(b)=\sum_{\beta_1<b_1}\sum_{\beta_2<b_2}\dots\sum_{\beta_{D-1}<b_{D-1}}\rho_t(\beta_1,\dots,\beta_{D-1})$$
satisfies $ \pi_t\equiv \psi_{\phi_t}$. Observe that for $(t,b)\in [N]^{(D)}$ with $t\in [N]$, we have
$$\phi(t,b)=\sum_{\alpha<t}\phi_{\alpha}(b).$$
Therefore, for $(t,u,a)\in [N]^{(2D)}$ with $t,u\in [N]$, we have
$$\psi_{\phi}(t,u,a)=\psi_{\phi(u,.)}(a)-\psi_{\phi(t,.)}(a)=\psi_{\phi(u,.)-\phi(t,.)}(a)=\sum_{t\leq\alpha<u}\psi_{\phi_{\alpha}}(a)=\sum_{t\leq\alpha<u} \pi_{\alpha}(a)= \pi(t,u,a).$$
In the first equality, we used \ref{property:1}, in the second and third, we used \ref{property:3}. This finishes the proof.
\end{proof}

For $k\in \{0,\dots,D\}$, we also define the evaluations $\psi^{(k)}_{\phi}:A^{(2D-k)}\rightarrow\mathbb{R}$ as follows. Let $(a,b)\in A^{(2D-k)}$ be an input, where $a\in A^{(k)}$ and $b\in A^{(2D-2k)}$, then
$$\psi^{(k)}_{\phi}(a,b):=\psi_{\phi(a,.)}(b)=\sum_{\eps\in \{0,1\}^{D-k}}(-1)^{\eps_1+\dots+\eps_{D-k}}\phi(a,(b_{2i-\eps_i})_{i\in [D-k]}).$$
These evaluations interpolate between $\psi_{\phi}$ and $\phi$ as we have  $\psi^{(0)}_{\phi}\equiv \psi_{\phi}$ and $\psi^{(D)}\equiv\phi$. Moreover, we have the following easy to check properties:
\begin{enumerate}[label={(\arabic*)}]
\setcounter{enumi}{6}
\item\label{property:7} If $\phi:A^{(D)}\rightarrow \mathbb{R}$ does not depend on the $\ell$-th coordinate, then $\psi_{\phi}^{(k)}\equiv 0$ for every $k<\ell$.
\item\label{property:8} If $(a,a_{k+1},a_{k+2},b)\in A^{(2D-k)}$ such that $a\in A^{(k)}$, $a_{k+1},a_{k+2}\in A$, and $b\in A^{(2D-2k-2)}$, then
$$\psi_{\phi}^{(k)}(a,a_{k+1},a_{k+2},b)=\psi_{\phi}^{(k+1)}(a,a_{k+2},b)-\psi_{\phi}^{(k+1)}(a,a_{k+1},b).$$\
\end{enumerate}
Here, \ref{property:7} follows from \ref{property:2}, and \ref{property:8} easily follows from the definition. 

\subsection{Exponentially separated functions on hypergraphs}

In this section, we present the main technical lemma underpinning the proof of \Cref{thm:main}. This lemma states that given $\phi:[N]^{(D)}\rightarrow\mathbb{R}$, we can find $\sigma\sim \phi$ that is exponentially separated on some large subset $S\subset [N]$. Moreover, the statement can be extended for a small set $\Phi$ of functions $\phi:[N]^{(D)}\rightarrow\mathbb{R}$ as well.

 \begin{lemma}\label{lemma:exp_sep_general}
Let $D,s$ be positive integer constants and $K=3D^2/2+D/2+1$. Let $\Phi$ be a set of $s$ functions $\phi:[N]^{(D)}\rightarrow \mathbb{R}$. Then there exists $S\subset [N]$ of size $n$  for some $N\leq \tw_K(n^{O(1)})$, and for every $\phi$, there exist $\sigma:S^{(D)}\rightarrow \mathbb{R}$ such that $\sigma\sim \phi$ on $S$ and $\sigma$ is exponentially separated.
\end{lemma}

\begin{proof}
We assume the $N$ is sufficiently large with respect to $D$ and $s$. Let $S_0=[N]$ and $\phi_0=\phi$ for every $\phi\in \Phi$. If $S_{\ell}\subset [N]$ and $\phi_{\ell}:S_{\ell}^{(D)}\rightarrow \mathbb{R}$ is already defined for some $\ell<D$, we define a subset $S_{\ell+1}$ and $\phi_{\ell+1}$ as follows.
\medskip

\noindent
\textbf{CONSTRUCTION OF $S_{\ell+1}$.} Define a coloring of $S_{\ell}^{(2D+1-\ell)}$ as follows. Let $F\in S_{\ell}^{(2D+1-\ell)}$, and write $F=(a,b_1,b_2,b_3,c)$ with $a\in S_{\ell}^{(\ell)}$, $b_1,b_2,b_3\in S_{\ell}$, and $c\in S_{\ell}^{(2D-2-2\ell)}$. Then we color $F$ with some  $\chi\in\{\typezero,\dots,\typefour\}^{\Phi}$ as follows. For every $\phi\in \Phi$, we compare the two numbers
$$X:=\psi^{(\ell)}_{\phi_{\ell}}(a,b_1,b_2,c)\mbox{\ \ \ \ and\ \ \ \ }Y:=\psi^{(\ell)}_{\phi_{\ell}}(a,b_2,b_3,c),$$
and we set $\chi_{\phi}$ to be
\begin{description}
    \item[\hspace{10pt} $\typeone$] if $0\leq X\leq Y$,
    \item[\hspace{10pt} $\typetwo$] if $0\leq Y\leq X$,
    \item[\hspace{10pt} $\typethree$] if $0\leq (-X)\leq (-Y)$,
    \item[\hspace{10pt} $\typefour$] if $0\leq (-Y)\leq (-X)$,
    \item[\hspace{10pt} $\typezero$] if $X$ and $Y$ has different signs.
\end{description}
 This gives a coloring of $S_{\ell}^{(2D+1-\ell)}$ by at most $5^s=O(1)$ colors, so by \Cref{thm:Ramsey}, $S_{\ell}$ contains a monochromatic set $T$ of some color $\chi$ such that $|S_{\ell}|\leq \tw_{2D+1-\ell}(O(|T|)$.  If $\chi_{\phi}=\typezero$ for some $\phi$, then $|S_{\ell+2}|\leq 2D+2$, so we can assume that $\typezero$ is not used in $\chi$. 

In what follows, fix $\phi\in \Phi$. Fix $a\in T^{(\ell)}$ and $c\in T^{(2D-2-2\ell)}$, and for  $b\in T$ with $a_{\ell}<b<c_1$, write
$$z_b(a,c)=z_b:=\psi_{\phi_{\ell}}^{(\ell+1)}(a,b,c).$$
Then the sequence $$\{z_b\}_{b\in T: a_{\ell}<b<c_1}$$ is shifted-exponential of type $\chi_{\phi}$. Indeed, if $b<b'$, then by \ref{property:8} we can write
$$z_{b'}-z_b=\psi^{(\ell+1)}_{\phi_{\ell}}(a,b',c)-\psi_{\phi_{\ell}}^{(\ell+1)}(a,b,c)=\psi_{\phi_{\ell}}^{(\ell)}(a,b,b',c),$$
so $T$ being monochromatic of color $\chi_{\phi}$ means that for every $b_1<b_2<b_3\in B$, the triple $(z_{b_1},z_{b_2},z_{b_3})$ has  type $\chi_{\phi}$. Hence, $\{z_b\}_{b}$ is indeed shifted-exponential of type $\chi_{\phi}$, which means that there exists $t(a,c)\in \mathbb{R}$  such that $$\left\{\psi_{\phi_{\ell}}^{(\ell+1)}(a,b,c)-t(a,c)\right\}_{b\in T:a_{\ell}<b<c_1}$$ is weakly-exponential of type $\chi_{\phi}$. We set $S_{\ell+1}:=T$.

\medskip

\noindent
\textbf{CONSTRUCTION OF $\phi_{\ell+1}$.} Our goal is to construct a function $\phi_{\ell+1}$ with the following two properties:
\begin{itemize}
\item[(i)] the sequence $$\left\{\psi_{\phi_{\ell+1}}^{(\ell+1)}(a,b,c)\right\}_{b\in S_{\ell+1}:a_{\ell}<b<c}$$
is weakly-exponential of type $\chi_{\phi}$ for every possible choice of $a\in T^{(\ell)}$ and $c\in T^{(2D-2-2\ell)}$;
\item[(ii)] $\psi_{\phi_{\ell+1}}^{(k)}\equiv \psi_{\phi_{k}}^{(k)}$ on $S_{\ell+1}$ for every $k\leq \ell$.
\end{itemize}
We show that by considering the values of $t(a,c)$ for various special values of $c$, we can construct such a $\phi_{\ell+1}$ using additive properties of $\psi$.

Fix $a\in T^{(\ell)}$ and write $q=D-1-\ell$ for simplicity. Also, after relabeling, we may assume that $T=[M]$ for some integer $M$. Given $d\in [M]^{(q)}$ such that $d_i<d_{i+1}-1$ and $d_{q}<M$, write $$d^+=(d_i,d_i+1)_{i=1,\dots,q}\in [M]^{(2q)}.$$
We consider the values $t(a,d^+)$ in order to construct $\phi_{\ell+1}$. Define the function $\rho_a:[M]^{(q)}\rightarrow\mathbb{R}$ such that $\rho_a(d)=t(a,d^+)$ if $d^+$ is defined, otherwise $\rho_a(d)=0$. We recall that the sequence 
$$Z_{a,d}:=\left\{\psi_{\phi_{\ell}}^{(\ell+1)}(a,b,d^+)-\rho_a(d)\right\}_{b:a_{\ell}<b<d_1}$$
is weakly-exponential of type $\chi_{\phi}$. For $b\in [a_{\ell}+1,M]$, define the function $\phi^*=\phi_{\ell}(a,b,.)$. Then by \ref{property:6}, for arbitrary $c\in [b+1,M]^{(2q)}$, we have
$$\psi_{\phi_{\ell}}^{(\ell+1)}(a,b,c)=\psi_{\phi^*}(c)=\sum_{c_1\leq d_1<c_2}\sum_{c_3\leq d_2<c_4}\dots\sum_{c_{2q-1}\leq d_s<c_{2q}}\psi_{\phi^*}((d_1,\dots,d_q)^+).$$
This motivates the definition $ \pi_a:[a_{\ell}+1,M]^{(2q)}\rightarrow\mathbb{R}$ such that
$$ \pi_a(c)=\sum_{c_1\leq d_1<c_2}\sum_{c_3\leq d_2<c_4}\dots\sum_{c_{2q-1}\leq d_s<c_{2q}}\rho_a(d_1,\dots,d_q),$$
because then we can write 
$$\psi_{\phi_{\ell}}^{(\ell+1)}(a,b,c)- \pi_a(c)=\sum_{c_1\leq d_1<c_2}\sum_{c_3\leq d_2<c_4}\dots\sum_{c_{2q-1}\leq d_s<c_{2q}}\left[\psi_{\phi_{\ell}}^{(\ell+1)}(a,b,(d_1,\dots,d_q)^+)-\rho_a(d_1,\dots,d_q)\right]$$
This means that the sequence
$$Z_{a,c}:=\left\{\psi_{\phi_{\ell}}^{(\ell+1)}(a,b,c)- \pi_a(c)\right\}_{b:a_{\ell}<b<c_1}$$
is the sum of the sequences $Z_{a,d}$ for a collection of pairs $(a,d)$. But each $Z_{a,d}$ is weakly-exponential of type $\chi_{\phi}$, so $Z_{a,c}$ is also weakly-exponential of type $\chi_{\phi}$. The advantage of the function $ \pi_a$ is that we can apply \Cref{lemma:phi_construction} to find $\xi_a:[a_{\ell}+1,M]^{(q)}\rightarrow \mathbb{R}$ such that $\psi_{\xi_a}\equiv \pi_a$. Finally, define $\theta:[M]^{(D)}\rightarrow \mathbb{R}$ such that $\theta(a,b,d):=\xi_a(d)$ (so in particular $\theta$ does not depend on $b$, the $(\ell+1)$-th coordinate), and set $\phi_{\ell+1}:=\phi_{\ell}-\theta$. With this definition, we have
$$\psi_{\phi_{\ell}}^{(\ell+1)}(a,b,c)- \pi_a(c)=\psi_{\phi_{\ell}}^{(\ell+1)}(a,b,c)-\psi_{\theta}^{(\ell+1)}(a,b,c)=\psi_{\phi_{\ell+1}}^{(\ell+1)}(a,b,c).$$
Therefore, the sequence
$$Z_{a,c}=\left\{\psi_{\phi_{\ell+1}}^{(\ell+1)}(a,b,c)\right\}_{b:a_{\ell}<b<c_1}$$
is weakly-exponential of type $\chi_{\phi}$. Also, as $\theta(a,b,c)$ does not depend on  the $(\ell+1)$-th coordinate, we have $\psi_{\theta}^{(k)}\equiv 0$ for every $k\leq \ell$ by \ref{property:7}, and thus $\psi_{\phi_{k}}^{(k)}\equiv \psi_{\phi_{\ell+1}}^{(k)}$. This also ensures that $\phi_{\ell+1}\sim\phi_{\ell}\sim\phi$. Therefore, (i) and (ii) are both satisfied. 

\medskip

\noindent
\textbf{EXPONENTIALLY SEPARATING NEIGHBORHOODS.}
Let $U=S_{D}$ and for every $\phi\in \Phi$, set $\sigma=\phi_{D}$. Our next goal is to find a large $V\subset U$ such that for any $k\in [D]$, $\phi\in \Phi$, and $(a,c)\in V^{(D-1)}$ with $a\in V^{(k-1)}$ and $c\in V^{(D-k)}$, the sequence $$\{\sigma(a,b,c)\}_{b\in V:a_{k-1}<b<c_1}$$ is weakly-exponential of type depending only on $k$ and $\phi$. In other words, any sequence we get by fixing all but the $k$-th coordinate for any $k\in [D]$ is weakly-exponential. 

In the remainder of the proof, we fix $\phi\in \Phi$ and write simply $\psi^{(k)}$ instead of $\psi_{\sigma}^{(k)}$. We record that for every $1\leq k\leq D$, $(a,c)\in U^{(2D-k-1)}$ with $a\in U^{(k-1)}$ and $c\in U^{(2D-2k)}$, the sequence
$$\{\psi^{(k)}(a,b,c)\}_{a_{k-1}<b<c_1}$$
is weakly-exponential of some type $\tau_k$. We say that a set $W\subset U$ has the \emph{$(k,m)$-property} for $1\leq m\leq k\leq D$ if for every $(a,c)\in W^{(2D-k-1)}$, with $a\in W^{(m-1)}$ and $c\in W^{(2D-k-m)}$, the sequence
$$\{\psi^{(k)}(a,b,c)\}_{b\in W:a_{m-1}<b<c_1}$$
is weakly-exponential of some type $\tau_{k,m}$. Thus, $U$ (and every subset of $U$) has the $(m,m)$-property for $m\in [D]$. Our goal can be reformulated as finding $V\subset U$ such that $V$ has the $(D,m)$-property for every $m\in [D]$. This suffices by recalling that $\psi^{(D)}\equiv \sigma$.

In what follows, we define a sequence of sets $U=U_0\supset U_{1}\supset\dots \supset U_{D-1}$ such that for every $\ell\in \{0,\dots,D-1\}$, $U_{\ell}$ has the $(\min\{D,m+\ell\},m)$-property for $m\in [D]$. This holds for $\ell=0$. Assume that $U_{\ell}$ is already defined satisfying required condition, then we define $U_{\ell+1}$ by removing the first and last, and every second element of $U_{\ell}$.
\begin{claim}
$U_{\ell+1}$ suffices.
\end{claim}
\begin{proof}
Fix some $m<D-\ell$ and write $k=m+\ell$. Consider $\psi^{(k+1)}$, and write the inputs of this function in the form 
$$(a,b,c,d,e)\mbox{ where }a\in U_{\ell}^{(m-1)}, b\in U_{\ell}, c\in U_{\ell}^{(\ell)}, d\in U_{\ell}, e\in U_{\ell}^{(2(D-k-1))}.$$ 
Let $b<b'$ and $d<d'$, and use \ref{property:8} to write the identities
$$\psi^{(k)}(a,b,c,d,d',e)=\psi^{(k+1)}(a,b,c,d',e)-\psi^{(k+1)}(a,b,c,d,e)$$
and
$$\psi^{(k)}(a,b',c,d,d',e)=\psi^{(k+1)}(a,b',c,d',e)-\psi^{(k+1)}(a,b',c,d,e).$$
For ease of notation, write
$$X_1=\psi^{(k)}(a,b,c,d,d',e), Y_1=\psi^{(k+1)}(a,b,c,d',e), Z_1=\psi^{(k+1)}(a,b,c,d,e)$$
and
$$X_2=\psi^{(k)}(a,b',c,d,d',e), Y_2=\psi^{(k+1)}(a,b',c,d',e), Z_2=\psi^{(k+1)}(a,b',c,d,e),$$
then $X_1=Y_1-Z_1$ and $X_2=Y_2-Z_2$. Using that $U_{\ell}$ has the $(k,k)$-property, we have for $i=1,2$ that 
\begin{alignat*}{2}
    &0\leq 2Z_i\leq Y_i&&\mbox{ if }\tau_{k}=\typeone,\\
    &0\leq 2Y_i\leq Z_i&&\mbox{ if }\tau_{k}=\typetwo,\\
    &0\leq (-2Z_i)\leq (-Y_i)&&\mbox{ if }\tau_{k}=\typethree,\\
    &0\leq (-2Y_i)\leq (-Z_i)&&\mbox{ if }\tau_{k}=\typefour.
\end{alignat*}
Therefore,
\begin{align*}
X_i\in \left[\frac{1}{2},1\right]Y_i\hspace{10pt}\mbox{ if }\tau_{k}\in \{\typeone,\typethree\},\\
X_i\in -\left[\frac{1}{2},1\right]Z_i\hspace{10pt}\mbox{ if }\tau_{k}\in \{\typetwo,\typefour\}.
\end{align*}
Next, observe that $U_{\ell}$ has the $(k,m)$-property. Hence, assuming that $b$ and $b'$ are not consecutive elements, we have
\begin{alignat*}{2}
    &0\leq 4X_1\leq X_2&&\mbox{ if }\tau_{k,m}=\typeone,\\
    &0\leq 4X_2\leq X_1&&\mbox{ if }\tau_{k,m}=\typetwo,\\
    &0\leq (-4X_1)\leq (-X_2)&&\mbox{ if }\tau_{k,m}=\typethree,\\
    &0\leq (-4X_2)\leq (-X_1)&&\mbox{ if }\tau_{k,m}=\typethree.
\end{alignat*}
We analyze these inequalities depending on the values of $(\tau_{k},\tau_{k,k})$. This presents 16 cases, of which 8 is not possible, as we will see. We present the detailed analysis of only a few of these cases, as most of the arguments are almost identical.
\begin{description}
    \item[($\typeone$,$\typeone$)] We have $X_i\in \left[\frac{1}{2},1\right]Y_i$ and $0\leq 4X_1\leq X_2$. Therefore, $$0\leq 2Y_1\leq 4X_1\leq X_2\leq Y_2.$$ In particular, $0\leq 2Y_1\leq Y_2$. As this is the case for every non-consecutive $b<b'$, we get that the sequence
    $$\left\{\psi^{(k+1)}(a,b,c,d',e)\right\}_{b\in U_{\ell+1}:a_{m-1}<b<c_1}$$
    is weakly-exponential of type $\tau_{k+1,\ell}=\typeone$. Moreover, as this is true for every choice of $d<d'$, we can choose $d'$ to be any element between $c$ and $e$, other than the successor of $c$ in $U_{\ell}$. In particular, if $a,c,e$ are all tuples in $U_{\ell+1}$, the successor of $c$ in $U_{\ell}$ is not in $U_{\ell+1}$. In conclusion, the sequence 
     $$\left\{\psi^{(k+1)}(a,b,c')\right\}_{b\in U_{\ell+1}:a_{m-1}<b<c'_1}$$
     is weakly-exponential of type $\typeone$ for every choice of $(a,c')\in U_{\ell+1}^{(2D-k-2)}$ with $a\in U_{\ell+1}^{(m-1)}$. Thus $U_{\ell+1}$ has property $(m+\ell+1,m)$.

     \item[($\typeone$,$\typetwo$)] We have $0\leq 2Y_2\leq 4X_2\leq X_1\leq Y_1$. Hence, we deduce that 
     $$\left\{\psi^{(k+1)}(a,b,c')\right\}_{b\in U_{\ell+1}:a_{m-1}<b<c'_1}$$
     is weakly-exponential of type $\typetwo$.
     \item[($\typeone$,$\typethree$)] This case is not possible. Indeed, we have $Y_i>0$ and $X_i\in [1/2,1]Y_i$, so $X_i>0$ as well. But this contradicts $\tau_{k,m}=\typethree$. All further ''Not possible'' cases are due to similar reason.
     \item[($\typeone$,$\typefour$)] Not possible.
     \item[($\typetwo$,$\typeone$)] Not possible.
     \item[($\typetwo$,$\typetwo$)] Not possible.
     \item[($\typetwo$,$\typethree$)] We have $0\leq 2Z_1\leq (-4X_1)\leq (-X_2)\leq Z_2$, so 
     $$\left\{\psi^{(k+1)}(a,b,c,d,e)\right\}_{b\in U_{\ell+1}:a_{m-1}<b<c_1}$$
     is weakly-exponential of type $\tau_{k+1,\ell}=\typeone$. As this is true for every $d<d'$, the only $d$ we are not allowed to chose is the predecessor of $e_1$. We conclude that 
     $$\left\{\psi^{(k+1)}(a,b,c')\right\}_{b\in U_{\ell+1}:a_{m-1}<b<c'_1}$$
     is weakly-exponential of type $\tau_{k+1,\ell}=\typeone$ for every choice of $(a,c')\in U_{\ell+1}^{(2D-k-2)}$ with $a\in U_{\ell+1}^{(m-1)}$. Thus $U_{\ell+1}$ has property $(m+\ell+1,m)$.
     \item[($\typetwo$,$\typefour$)]  $0\leq 2Z_2\leq (-4X_2)\leq (-X_1)\leq Z_1$, so $\tau_{k+1,\ell}=\typetwo$.
     \item[($\typethree$,$\typeone$)] Not possible.
     \item[($\typethree$,$\typetwo$)] Not possible.
     \item[($\typethree$,$\typethree$)]  $0\leq (-2Y_1)\leq (-4X_1)\leq (-X_2)\leq (-Y_2)$, so $\tau_{k+1,\ell}=\typethree$.
     \item[($\typethree$,$\typefour$)]  $0\leq (-2Y_2)\leq (-4X_2)\leq (-X_1)\leq (-Y_1)$, so $\tau_{k+1,\ell}=\typefour$.
     \item[($\typefour$,$\typeone$)]  $0\leq (-2Z_1)\leq 4X_1\leq X_2\leq (-Z_1)$, so $\tau_{k+1,\ell}=\typethree$.
     \item[($\typefour$,$\typetwo$)]  $0\leq (-2Z_2)\leq 4X_2\leq X_1\leq (-Z_1)$, so $\tau_{k+1,\ell}=\typefour$.
     \item[($\typefour$,$\typethree$)] Not possible.
     \item[($\typefour$,$\typefour$)] Not possible.
\end{description}
\end{proof}

Set $V:=U_{D-1}$, then $|V|=\Omega(|U|)$ and $V$ has the $(D,m)$-property for every $m\in [D]$. Equivalently, for any $k\in [D]$ and $(a,c)\in V^{(D-1)}$ with $a\in V^{(k-1)}$ and $c\in V^{(D-k)}$, the sequence $$\left\{\sigma(a,b,c)\right\}_{b\in V:a_{k-1}<b<c_1}$$ is weakly-exponential. We refer to this as the \emph{weak separation property}.

\medskip

\noindent
\textbf{EXPONENTIALLY SEPARATING EVERYTHING.} We are almost done, all it remains is to find a large subset $S\subset V$ such that $\sigma$ is exponentially separated on $S^{(D)}$. The existence of such an $S$ of size $|S|=\Omega(|V|^{1/(2D-1)})$ follows almost immediately from \Cref{lemma:rainbow}. Indeed, consider the function $\gamma:V^{(D)}\rightarrow\mathbb{R}^{\Phi}$ defined as $\gamma(a)_{\phi}:=\log_2 |\sigma(a)|$. Then $\gamma$ satisfies the following property. Given $a\in V^{(D)}$, write $\langle a\rangle=\{a_1,\dots,a_D\}$. Then for every $a\in V^{(D)}$ and $\phi\in \Phi$, there are at most $2D^2$ other $b\in V^{(D)}$ such that $|\langle a\rangle\cap \langle b\rangle|=D-1$ and $|\gamma(a)_{\phi}-\gamma(b)_{\phi}|< 1$. This is true because if $I,J\in [D]^{(D-1)}$, then there are at most two such choices for $b$ such that $a_I=b_J$  by the weak separation property.  Hence, we can apply \Cref{lemma:rainbow} with $s=2D^2$ to get a set $S$ of size $|S|=\Omega(|V|^{1/(2D-1)})$ such that any two distinct $a,b\in S^{(D)}$ satisfies $|\gamma(a)_{\phi}-\gamma(b)_{\phi}|\geq 1$ for every $\phi\in \Phi$. But this implies that $\sigma$ is exponentially separated on $S^{(D)}$ for every $\phi\in \Phi$.

To finish the proof, we bound the size of $S$. We have $|S_0|=N$ and $|S_{\ell}|\leq \tw_{2D+1-\ell}(O(|S_{\ell+1}|))$ for $\ell=0,\dots,D-1$. Therefore, $U=S_D$ satisfies $N\leq \tw_{K}(O(|U|))$ with 
$$K=\left(\sum_{\ell=0}^{D-1}2D+1-\ell\right)-(D-1)=3D^2/2+D/2+1.$$ Moreover, $|V|=\Omega(|U|)$ and $|S|=\Omega(|V|^{1/(2D-1)})$. Thus, $N\leq \tw_{K}(|S|^{O(1)})$, finishing the proof.
\end{proof}

\subsection{Proof of the Main theorem}\label{sect:maintheorem}

In this section, we prove \Cref{thm:main}. In particular, we prove a more general result, which we prepare with two definitions.

\begin{definition}
    A function $f:A^{(r)}\rightarrow \mathbb{R}$ is \textbf{$D$-dependent} if there exist functions $\phi_I:A^{(D)}\rightarrow \mathbb{R}$ for $I\in [r]^{(D)}$ such that for every $a\in A^{(r)}$,
    $$f(a)=\sum_{I\in [r]^{(D)}}\phi_I(a_I).$$
\end{definition}

In other words, $f:A^{(r)}\rightarrow \mathbb{R}$ is $D$-dependent if it is the linear combinations of functions that depend on at most $D$ of the input variables. This extends the class of polynomials of degree $D$.

\begin{claim}
If $f:(\mathbb{R}^d)^r\rightarrow \mathbb{R}$ is polynomial of degree at most $D$, and $V\subset \mathbb{R}^d$ is an ordered set, then $f$ restricted to $V^{(r)}$ is $D$-dependent.
\end{claim}

\begin{proof}
  Fix an arbitrary ordering of the elements of $[r]^{(D)}$.  Let $x\in (\mathbb{R}^d)^r$, then each monomial in $f(x)$ has the form 
$$c\prod_{i=1}^r\prod_{j=1}^dx_{i,j}^{\alpha_{i,j}}$$
with $\sum_{i=1}^r\sum_{j=1}^d\alpha_{i,j}\leq D$. For $I\in [r]^{(D)}$, add this monomial to the function $f_I: (\mathbb{R}^d)^r\rightarrow\mathbb{R}$ if $I$ is the first $D$-tuple in the ordering that contains all $i\in [r]$ such that $\alpha_{i,1}+\dots+\alpha_{i,d}\geq 1$. As $f$ has degree at most $D$, each monomial is added to some $f_I$, so 
$$f(x)=\sum_{I\in [r]^{(D)}}f_I(x).$$
 Also, $f_I$ does not depend on $x_j$ if $j\not\in I$, so we may write $f_I(x)=\phi_I(x_I)$ with suitable function $\phi_I:V^{(D)}\rightarrow \mathbb{R}$, $x\in V^{(r)}$. This finishes the proof.
\end{proof}

\noindent
Next, using the notion of $D$-dependent functions, we extend the class of semialgebraic hypergraphs of complexity $(d,D,m)$ in the following natural manner.

\begin{definition}
    Let $A$ be an ordered set, let $f_1,\dots,f_m:A^{(r)}\rightarrow \mathbb{R}$ be $D$-dependent functions, and let $\Gamma: \{\texttt{false},\texttt{true}\}^{m}\rightarrow \{\texttt{false},\texttt{true}\}$ be a Boolean formula. Define the $r$-uniform  hypergraph $H$ on vertex set $A$ such that $a\in A^{(r)}$ is an edge if and only if
    $$\Gamma\left(f_1(a)\leq 0,\dots,f_m(a)\leq 0\right)=\texttt{true}.$$
    Then $H$, and any hypergraph isomorphic to $H$, is \textbf{quasi-algebraic} of complexity $(D,m)$.
\end{definition}

\noindent
Clearly, a semialgebraic hypergraph of complexity $(d,D,m)$ is also quasi-algebraic of complexity $(D,m)$. Therefore, the following theorem implies \Cref{thm:main}. The rest of the section is devoted to the proof of this theorem.

\begin{theorem}
     Let $r,D,m$ be positive integer constants. Let $H$ be an $r$-uniform quasi-algebraic hypergraph of complexity $(D,m)$ on $N$ vertices. Then $H$ contains a clique or an independent set of size $n$, where 
     $$N\leq \tw_{K}(O(n))\mbox{ for }K=\frac{D(D^2+6D+1)}{2}.$$
\end{theorem}

\begin{proof}
Let $A$ be the vertex set of $H$, let $F$ be a set of $D$-dependent functions $f:A^{(r)}\rightarrow \mathbb{R}$ of size $m$, and let $\Gamma: \{\texttt{false},\texttt{true}\}^{F}\rightarrow \{\texttt{false},\texttt{true}\}$ be a Boolean expression such that $a\in A^{(r)}$ is an edge of $H$ if and only if
$$\Gamma\left[(f(a)\leq 0)_{f\in F}\right]=\texttt{true}.$$

Fix some $f\in F$. Using that $f$ is $D$-dependent, for every $I\in [r]^{(D)}$, there exists a function  $\phi_{f,I}:A^{(D)}\rightarrow \mathbb{R}$  such that for every $a\in A^{(r)}$,  we have
$$f(a)=\sum_{I\in [r]^{(D)}}\phi_{f,I}(a_I).$$

In what follows, we recursively define a sequence of sets $A=S_0\supset S_1\supset\dots\supset S_D$, and for every $f\in F$ and $I\in [r]^{(\leq D)}$, two functions $\sigma_{f,I},\phi_{f,I}:A^{(|I|)}\rightarrow \mathbb{R}$ such that the following properties hold. For every $f\in F$ and $k\in \{0,\dots,D\}$,
\begin{enumerate}
    \item     $f(a)=\displaystyle\sum_{I\in [r]^{(D-k)}}\phi_{f,I}(a_I)+\sum_{\ell=D-k+1}^D\sum_{I\in [r]^{(\ell)}}\sigma_{f,I}(a_I),$
    \item if $k\geq 1$ and $I\in [r]^{(D-k+1)}$, then $\sigma_{f,I}$ is exponentially separated on $S_{k+1}$.
\end{enumerate}
Assume that for some $k\in \{0,\dots,D-1\}$, $S_k$ and functions $\phi_{f,I}$ are already defined for $I\in [r]^{(D-k)}$. Then we define $\sigma_{f,I}$ for $I\in [r]^{(D-k)}$, $\phi_{f,J}$ for $J\in [r]^{(D-k-1)}$, and $S_{k+1}$ as follows. Let 
$$\Phi=\{\phi_{f,I}:f\in F,I\in [r]^{(D-k)}\},$$
and apply \Cref{lemma:exp_sep_general} with the set of at most $s=m\binom{r}{D-k}$ functions $\Phi$ acting on $S_{k}^{(D-k)}$. Then there exists $S_{k+1}\subset S_k$ such that $|S_k|\leq \tw_{K}(|S_{k+1}|^{O(1)})$ with $K=3(D-k)^2/2+(D-k)/2+1$, and for every $\phi_{f,I}\in \Phi$, there exists $\sigma_{f,I}$ such that $\phi_{f,I}\sim \sigma_{f,I}$ and $\sigma_{f,I}$ is exponentially separated on $S_{k+1}$. In particular, there exist functions $\rho_{f,I,J}:S_{k+1}^{(D-k-1)}\rightarrow \mathbb{R}$ for $J\in I^{(D-k-1)}$ such that for every $a\in S_{k+1}^{(r)}$,
$$\phi_{f,I}(a_I)=\sigma_{f,I}(a_I)+\sum_{J\in I^{(D-k-1)}}\rho_{f,I,J}(a_J).$$
Then for $J\in [r]^{(D-k-1)}$, we define
$$\phi_{f,J}=\sum_{I\in [r]^{(D-k)}, J\subset I}\rho_{f,I,J}.$$
This concludes the definition of $S_{k+1}$ and functions $\sigma_{f,I}$ and $\phi_{f,J}$. It is easy to show by induction, that 1. and 2. are satisfied. We note that $\phi_{f,\emptyset}$ denotes a constant, which we view as a function $\phi_{f,\emptyset}()$, and we simply set $\sigma_{f,\emptyset}:=\phi_{f,\emptyset}$.

Set $S=S_{D}$, then we have the following two properties satisfied: for every $f\in F$,
\begin{enumerate}
    \item 
    $f(a)=\displaystyle\sum_{I\in [r]^{(\leq D)}}\sigma_{f,I}(a_I),$
    \item $\sigma_{f,I}$ is exponentially separated on $S$ for every $I\in [r]^{(\leq D)}$.
\end{enumerate}

Next, define a coloring of $S^{(2D)}$ as follows. Given $a\in S^{(2D)}$, let $Q_a$ be the sequence of all numbers $|\sigma_{f,I}(a_J)|$ for $f\in F$, $I\in [r]^{(\leq D)}$ and $J\subset [2D]^{(|I|)}$. Color the edge $a$ with color $0$ if there exists $f$, and $(I,J)$ and $(I',J')$ with $J\neq J'$ such that 
$$2r^{-D} |\sigma_{f,I}(a_J)|<|\sigma_{f,I'}(a_{J'})|<2r^{D}|\sigma_{f,I}(a_J)|.$$ 
If $a$ is not colored 0, then order the elements of $Q_a$ in increasing order, and let $R_a$ be the resulting sequence. For each entry $R_{a,i}$, write down the triple $(f,I,J)$ such that $R_{a,i}=|\sigma_{f,I}(a_J)|$, this gives an ordering $\chi$ of all possible triples $(f,I,J)$. Color $a$ with color $\chi$. This gives a coloring of $S^{(2D)}$ with $O(1)$ colors. Applying \Cref{thm:Ramsey}, there exists monochromatic $T\subset S$ such that $|S|\leq \tw_{2D}(O(|T|)$. 

\begin{claim}
 $T$ is not colored with 0, assuming $|T|$ is sufficiently large.
 \end{claim}
 \begin{proof}
Let $u=\lceil10D\log_2 r\rceil$. Assume to the contrary that $T$ is monochromatic of color $0$. Then, under the assumption that $|T|$ is sufficiently large, we can find $T_0\subset T$ of size $2Du$, and a pair of triplets $(f,I,J)$ and $(f,I',J')$ with $J\neq J'$ such that $2r^{-D}|\sigma_{f,I}(a_J)|< |\sigma_{f,I'}(a_{J'})|< 2r^{D}|\sigma_{f,I}(a_J)|$ for every $a\in T_0^{(2D)}$. This follows from a simple application of Ramsey's theorem. Relabel the elements of $T_{0}$ such that $T_{0}=[2Du]$, and let $a=(u,2u,\dots,2Du)\in T_0^{(2D)}$. Without loss of generality, $|J|\geq |J'|$, then there exists $j\in J\setminus J'$. Writing $a^{(p)}\in T_0^{(2D)}$ for the element we get by changing $a_{j}$ to $a_{j}-p$ in $a$, we have that for every $0\leq p<u$, $|\sigma_{f,I'}(a^{(p)}_{J'})|=|\sigma_{f,I'}(a_{J'})|$. On the other hand, the sequence $(|\sigma_{f,I}(a^{(p)}_J)|)_{p=0,\dots,u-1}$ is exponentially separated, so it contains two elements, whose ratio is larger than $2^{u-1}>4r^{2D}$. But then there is $p$ such that   $2r^{-D}|\sigma_{f,I}(a^{(p)}_J)|\geq |\sigma_{f,I'}(a^{(p)}_{J'})|$ or $2r^{D}|\sigma_{f,I}(a^{(p)}_J)|\leq |\sigma_{f,I'}(a^{(p)}_{J'})|$, contradiction.
\end{proof}

Therefore, we may assume that $T$ is monochromatic of some color $\chi\neq 0$. Let  $f\in F$ and $x\in T^{(r)}$, and consider
$$f(x)=\sum_{I\in [r]^{(\leq D)}}\sigma_{f,I}(x_I).$$
Assume that $I_0^{f,x}=I_0\in [r]^{(\leq D)}$ maximizes $|\sigma_{f,I}(x_I)|$, then using that $T$ is not of color $0$, we have $$r^{2D}|\sigma_{f,I}(x_I)|\leq |\sigma_{f,I_0}(x_{I_0})|$$ for every $I\neq I_0$. But as the sum contains $|[r]^{(\leq D)}|<2r^{D}$ elements, this means that the sign of $f(x)$ is the same as the sign of $\sigma_{f,I_0}(x_{I_0})$. The key observation is that $I_0$ does not depend on the choice of $x$. This follows from the more general statement that the order of the elements of the sequence $(\sigma_{f,I}(x_I))_{I\in [r]^{(\leq D)}}$ is completely determined by the color $\chi$.  Indeed, let $I,I'\in [r]^{(\leq D)}$ such that $I\neq I'$, and choose a $(2D)$-tuple $a\in T^{(2D)}$ whose first $|[I]\cup [I']|$ elements are the elements of $[I]\cup [I']$. Then $\sigma_{f,I}(x_I)=\sigma_{f,I}(a_J)$ and $\sigma_{f,I'}(x_{I'})=\sigma_{f,I'}(a_{J'})$  for some $J\neq J'\subset [2D]$ that depend only on $I$ and $I'$. As $a$ is colored by $\chi$, the relative order of $|\sigma_{f,I}(a_{J})|$ and $|\sigma_{f,I'}(a_{J'})|$ is completely determined by $\chi$.

To summarize, we have shown that for every $f\in F$, there exists $I_0^f\in [r]^{(\leq D)}$ such that the sign of $f(a)$ is the same as the sign of $\sigma_{f,I_0}(a_{I_0})$. Moreover, using that $\sigma_{f,I_0}$ is exponentially separated on $T$, this tells us that $f(a)$ has the same sign for every choice of $a$. Thus, 
$$\{f(a)\leq 0\}_{f\in F}\in \{\texttt{false},\texttt{true}\}^{F}$$
is constant on $T$, which means that $H[T]$ is a clique or an independent set.

We finish the proof by bounding the size of $T$. We have $|S_0|=N$ and $|S_{k}|\leq \tw_{K_k}(|S_{k+1}|^{O(1)})$ with $K_k=3(D-k)^2/2+(D-k)/2+1$ for $k=0,\dots,D-1$. Therefore, $N\leq \tw_{L}(|S|^{O(1)})$ with 
$$L=-(D-1)+\sum_{k=0}^{D-1}\frac{3(D-k)^2}{2}+\frac{D-k}{2}+1=\frac{D(D+1)^2}{2}+1.$$
Finally, we have $|S|\leq \tw_{2D}(O(|T|))$, which gives $N\leq \tw_{K}(O(|T|))$ with $K=\frac{D(D^2+6D+1)}{2}$.
\end{proof}

We observe that the bound $K=\frac{D(D^2+6D+1)}{2}$ satisfies $K<3D^3$ if $D\geq 2$. For $D=1$, we know the stronger bound $N\leq 2^{n^{O(1)}}$ from \cite{JinTomon}, so \Cref{thm:main} holds for all values of $D$.




\section*{Acknowledgments}
Both authors acknowledge the support of the Swedish Research Council grant VR 2023-03375.

\noindent
We would like to thank Zhihan Jin and J\'anos Pach for valuable discussions.

\end{document}